\documentclass[11pt]{amsart}
\usepackage{verbatim} 
\usepackage{arydshln} 
\usepackage{rotating}
\usepackage{amsmath}
\usepackage{amssymb}
\usepackage{amscd}
\usepackage{enumerate}
\usepackage{color}
\usepackage{bm}
\usepackage{easybmat} 
\usepackage{graphicx}
\usepackage{subfigure}

\newtheorem{Thm}{Theorem}[section]

\newtheorem{Lemma}[Thm]{Lemma}
\newtheorem{Prop}[Thm]{Proposition}

\theoremstyle{definition}

\newtheorem{Rmk}[Thm]{Remark}

\newtheorem*{notation}{Notation}

\numberwithin{equation}{section}

\def\bbc{\mathbb{C}}

\def\bbn{\mathbb{N}}

\def\bbz{\mathbb{Z}}

\def\calH{\mathcal{H}}

\def\lra{\longrightarrow}

\def\x{\times}

\def\aut{\mathrm{Aut}}

\def\id{\mathrm{id}}

\def\endo{\mathrm{Endo}}

\def\Sol{\mathrm{Sol}}

\def\inn{\mathrm{Inn}}

\def\GR{\mathrm{GR}}

\def\Sym{\mathsf{Sym}}
\def\FSym{\mathsf{FSym}}
\def\FAlt{\mathsf{FAlt}}
\def\res{\mathrm{res}}
\def\supp{\mathrm{supp}}

\def\boxit#1{\vbox{\hrule\hbox{\vrule\kern3pt
     \vbox{\kern3pt#1\kern3pt}\kern3pt\vrule}\hrule}}
\def\sqbox#1{
\gdef\boxcont{\rm #1}
  \setbox4=\vbox{\hsize 30pc \noindent \strut \boxcont\strut}
  \par\centerline{\boxit{\box4}}\par}

\def\colr#1{\textcolor{red}{#1}}

\begin{document}
\title{The $R_{\infty}$ property for Houghton's groups}
\author{Jang Hyun Jo}
\author{Jong Bum Lee}
\address{Department of Mathematics, Sogang University, Seoul 121-742, KOREA}
\email{jhjo@sogang.ac.kr}
\email{jlee@sogang.ac.kr}

\author{Sang Rae Lee}
\address
{Department of Mathematics, Texas A\&M University, College Station, Texas 77843, USA}
\email{srlee@math.tamu.edu}

\subjclass[2000]{20E45, 20E36, 55M20}%
\keywords{Houghton's group, $R_\infty$ property, Reidemeister number}

\begin{abstract}
We study twisted conjugacy classes of a family of groups which are called Houghton's groups $\calH_n$ ($n \in\bbn$), the group of translations of $n$ rays of discrete points at infinity.
We prove that the Houghton's groups $\calH_n$ have the $R_\infty$ property for all $n\in \bbn$.
\end{abstract}
\date{\today}
\maketitle



\section{Introduction}\label{sec:introd}

Let $G$ be a group and $\varphi:G\to G$ be a group endomorphism.
We define an equivalence relation $\sim$ on $G$, called the Reidemeister action by $\varphi$, by
$$
a\sim b \Leftrightarrow b=ha\varphi(h)^{-1} \text{ for some }h\in G.
$$
The equivalence classes are called \emph{twisted conjugacy classes} or \emph{Reidemeister classes} and $R[\varphi]$ denotes the set of twisted conjugacy classes.
The \emph{Reidemeister number} $R(\varphi)$ of $\varphi$ is defined to be the cardinality of $R[\varphi]$.
We say that $G$ has the \emph{$R_\infty$ property} if $R(\varphi)=\infty$ for every automorphism $\varphi:G\to G$.

In 1994, Fel'shtyn and Hill \cite{FH} conjectured that any injective endomorphism $\varphi$ of a finitely generated group $G$ with exponential growth would have infinite Reidemeister number. Levitt and Lustig (\cite{LL}), and Fel'shtyn (\cite{F}) showed that the conjecture holds for
automorphisms when $G$ is Gromov hyperbolic. However, in 2003, the conjecture was answered negatively by Gon\c{c}alves and Wong \cite{GW1} who gave examples of groups which do not have the $R_\infty$ property.
Since then, groups with the $R_{\infty}$ property have been known including Baumslag-Solitar groups, lamplighter groups, Thompson's group $F$, Grigorchuk group, mapping class groups, relatively hyperbolic groups, and some linear groups (see \cite{BFG, DG, FG, GK, GS1,GS2, GS3, GW2, HL, KW, STW} and references therein).
For a topological consequence of the $R_{\infty}$ property, see \cite{GW2, KW, STW}.
In this article we show the following.

\begin{Thm}\label{thm:main}
The Houghton's groups $\calH_n$ have the $R_\infty$ property for all $n\in \bbn$.
\end{Thm}

It is shown that the conjugacy problem(\cite{ABM}) and the twisted conjugacy problem(\cite{C}) of $\calH_n$ are solvable for $n \geq 2$. In 2010, Gon\c{c}alves and Kochloukova \cite{GK} proved that there is a finite index subgroup $H$ of $\mathrm{Aut}(\calH_n)$ such that $R(\varphi) = \infty$ for $\varphi \in H$ provided $n\geq 2$. Recently the structure of $\aut(\calH_n)$ is known from \cite{BCMR} (see Theorem~\ref{auto} below). In \cite{GS3}, Gon\c{c}alves and Sankaran have studied also the $R_\infty$ property of Houghton's groups.

In this paper we use simple but useful observations of the Reidmeister numbers and the structure of $\aut(\calH_n)$ to find equivalent conditions for two elements of $\calH_n$ to determine the same twisted conjugacy class under mild assumptions.
In Section~\ref{sec:Houghton}, we will review definition and some facts about Houghton's groups $\calH_n$
which are necessary mainly to the study of Reidemeister numbers for $\calH_n$.
In Section~\ref{sec:R number}, we prove our main result for $n\geq 2$. The case of $n=1$ is discussed in Section~\ref{sec:H_1}.

\section{Houghton's groups $\calH_n$}\label{sec:Houghton}
In this paper we use the following notational conventions. All bijections (or permutations) act on the right unless otherwise specified. Consequently $gh$ means $g$ followed by $h$. The conjugation by $g$ is denoted by $\mu(g)$, $h^g = g^{-1}hg = :\mu(g)(h)$, and the commutator is defined by $[g,h] = ghg^{-1}h^{-1}$.

Our basic references are \cite{H,SR} for Houghton's groups and \cite{BCMR} for their automorphism groups.
Fix an integer $n\ge1$. For each $k$ with $1\le k\le n$, let
$$
R_k=\left\{me^{i\theta}\in\bbc\mid m\in\bbn,\ \theta=\tfrac{\pi}{2}+(k-1)\tfrac{2\pi}{n} \right\}
$$
and let $X_n=\bigcup_{k=1}^n R_k$ be the disjoint union of $n$ copies of $\bbn$,
each arranged along a ray emanating from the origin in the plane.
We shall use the notation $\{1,\cdots,n\}\x\bbn$ for $X_n$,
letting $(k,p)$ denote the point of $R_k$ with distance $p$ from the origin.

A bijection $g:X_n\to X_n$ is called an \emph{eventual translation}
if the following holds:
\begin{quote}
There exist an $n$-tuple $(m_1,\cdots,m_n)\in\bbz^n$ and a finite set $K_g\subset X_n$ such that
$$
(k,p)\cdot g:=(k,p+m_k)\quad \forall (k,p)\in X_n-K_g.
$$
\end{quote}
An eventual translation acts as a translation on each ray outside a finite set.
For each $n \in \bbn$ the \emph{Houghton's group} $\calH_n$ is defined to be the group of all eventual translations of $X_n$.

Let $g_i$ be the translation on the ray of $R_1\cup R_{i+1}$ by $1$ for $1\le i\le n-1$.
Namely,
$$
(j,p)\cdot g_i=\begin{cases}
(1,p-1)&\text{if $j=1$ and $p\ge2$,}\\
(i+1,1)&\text{if $(j,p)=(1,1)$,}\\
(i+1,p+1)&\text{if $j=i+1$,}\\
(j,p)&\text{otherwise.}\end{cases}
$$

\begin{figure}[h]\label{fig:example}
    \includegraphics[width=1.0\textwidth]{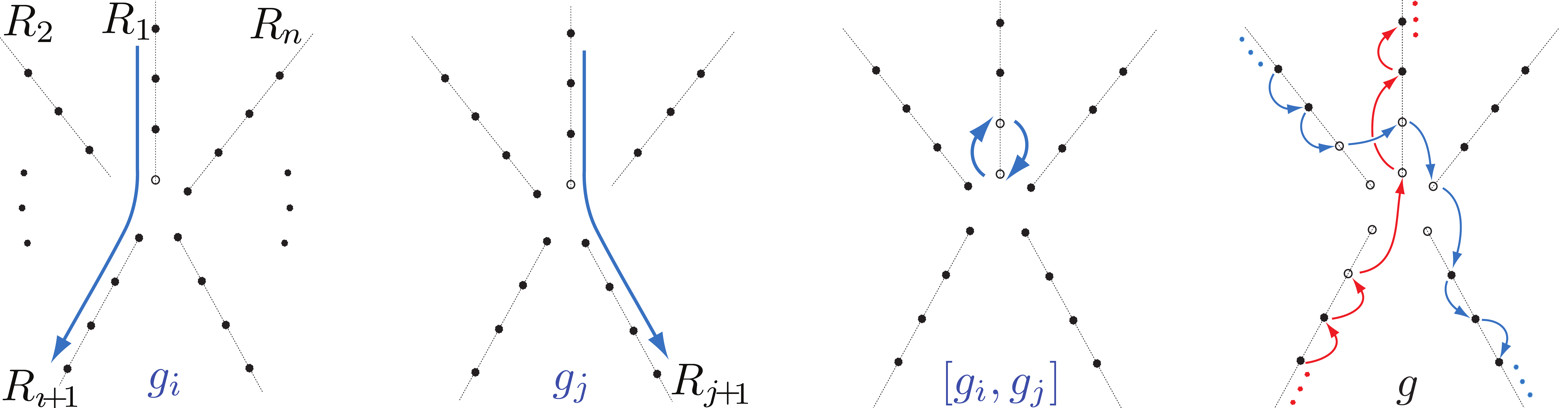}
    \caption{Some examples of $\calH_n$}.
\end{figure}

Figure \ref{fig:example} illustrates some examples of elements of $\calH_n$, where points which do not involve arrows are meant to be fixed. 
Finite sets $K_{g_i}$
and $K_{g_j}$ are singleton sets. The commutator $[g_i,g_j]$ of two distinct elements $g_i$ and $g_j$ is the transposition exchanging $(1,1)$ and $(1,2)$.
We will denote this transposition by $\alpha$. The last
element $g$ is rather generic and $K_g$ consists of eight points.
Johnson provided a finite presentation for $\calH_3$ in \cite{J} and the third author gave a finite presentation for $\calH_n$ with $n \geq 3$ in \cite{SR} as follows:
\begin{Thm}[{\cite[Theorem~C]{SR}}]\label{C}
For $n\ge3$, $\calH_n$ is generated by $g_1,\cdots,g_{n-1},\alpha$ with relations
$$
\alpha^2=1,\
(\alpha\alpha^{g_1})^3=1,\
[\alpha,\alpha^{g_1^2}]=1,\
\alpha=[g_i,g_j],\
\alpha^{g_i^{-1}}=\alpha^{g_j^{-1}}\
$$
for $1\le i\ne j\le n-1$.
\end{Thm}

From the definition of Houghton's groups, the assignment $g\in\calH_n\mapsto (m_1,\cdots,m_n)\in\bbz^n$
defines a homomorphism $\pi=(\pi_1,\cdots,\pi_n):\calH_n\to\bbz^n$. Then we have:

\begin{Lemma}[{\cite[Lemma~2.3]{SR}}]
For $n\ge3$, we have $\ker\pi
=[\calH_n,\calH_n]$.
\end{Lemma}

Note that $\pi(g_i)\in\bbz^n$ has only two nonzero values $-1$ and $1$,
$$
\pi(g_i)=(-1,0,\cdots,0,1,0,\cdots,0)
$$
where $1$ occurs in the $(i+1)$st component.
Since the image of $\calH_n$ under $\pi$ is generated by those elements,
we have that
$$
\pi(\calH_n)=\left\{(m_1,\cdots,m_n)\in\bbz^n\mid \sum_{i=1}^n m_i=0\right\},
$$
which is isomorphic to the free Abelian group of rank $n-1$.
Consequently, $\calH_n$ ($n\ge3$) fits in the following short exact sequence
$$
1\lra\calH_n'=[\calH_n,\calH_n]\lra\calH_n\buildrel{\pi}\over\lra\bbz^{n-1}\lra1.
$$
The above abelianization, first observed by C. H. Houghton in \cite{H}, is the characteristic property of $\{\calH_n\}$ for which he introduced those groups in the same paper.
We may regard $\pi$ as a homomorphism $\calH_n\to\bbz^n\to\bbz^{n-1}$ given by
$$
\pi:g_i\mapsto(-1,0,\cdots,0,1,0,\cdots,0)\mapsto(0,\cdots,0,1,0,\cdots,0).
$$
In particular, $\pi(g_1),\cdots,\pi(g_{n-1})$ form a set of free generators for $\bbz^{n-1}$.

As definition, $\calH_1$ is the symmetric group itself on $X_1$ with finite support,
which is not finitely generated.
Furthermore, $\calH_2$ is
$$
\calH_2=\langle g_1,\alpha \mid \alpha^2=1, (\alpha\alpha^{g_1})^3=1,
[\alpha,\alpha^{g_1^k}]=1 \text{ for all } |k|>1\rangle,
$$
which is finitely generated, but not finitely presented. It is not difficult to see that $\calH_2'=\FAlt_2$.

\begin{notation}
\begin{align*}
&\Sym_n= \text{ the full symmetric group of $X_n$,}\\
&\FSym_n= \text{ the symmetric group of $X_n$ with finite support,}\\
&\FAlt_n= \text{ the alternating group of $X_n$ with finite support.}
\end{align*}
\end{notation}
\noindent


For each $n$ the group $\FAlt_n$ can be seen as the kernel of the sign homomorphism $\FSym_n\to\{\pm1\}$. The following fact is necessary for our discussion, see  \cite{DM}.
\begin{Rmk}\label{prop:conjugation}
For any $\sigma\in\Sym_n$, the conjugation by $\sigma$ induces automorphisms
$\mu(\sigma):\FSym_n\to\FSym_n$ and $\mu(\sigma):\FAlt_n\to\FAlt_n$.
Then $\mu:\Sym_n\to\aut(\FAlt_n)$ and $\mu:\Sym_n\to\aut(\FSym_n)$ are isomorphisms.
\end{Rmk}

Every automorphism of $\calH_n$ restricts to an automorphism of the characteristic subgroup $\calH_n''=[\FSym_{n},\FSym_{n}]=\FAlt_n$, which induces a homomorphism $\res: \aut(\calH_n)\to\aut(\FAlt_n)$. One can show this map is injective by using the fact that $\FAlt_n$ is generated by $3$-cycles. The embedding $$
\mathrm{Res}:\aut(\calH_n)\buildrel\res\over\lra \aut(\FAlt_n)\buildrel\mu^{-1}\over\lra \Sym_n
$$ implies that each automorphism of $\calH_n$ is given by a conjugation of an element in $\Sym_n$. Moreover the composition preserves the normality $\calH_n = \inn(\calH_n) \lhd \aut(\calH_n)$.

\begin{Prop}[{\cite[Proposition~2.1]{BCMR}}]
For $n\ge1$, the automorphism group $\aut(\calH_n)$ is isomorphic to
the normalizer of $\calH_n$ in the group $\Sym_n$.
\end{Prop}

We need an explicit description for the normalizer $N_{\Sym_n}(\calH_n)$ to study $\aut(\calH_n)$.
Consider an element $\sigma_{ij}\in \Sym_n$ for $1\leq i\neq j\leq n$ defined by
$$
(\ell,p)\cdot\sigma_{ij}=
    \begin{cases}
        (j,p) &\text{if } \ell=i\\
        (i,p) &\text{if } \ell=j\\
        (\ell,p) &\text{otherwise}
    \end{cases}
$$ for all $p\in \bbn$. Each element $\sigma_{ij}$ defines a transposition on $n$ rays isometrically. The subgroup of $\Sym_n$ generated by all $\sigma_{ij}$ is isomorphic to the symmetric group $\Sigma_n$ on the $n$ rays. Note that $\Sigma_n$ acts on $\calH_n$ by conjugation. One can show that $N_{\Sym_n}(\calH_n)$ coincides with $\calH_n\rtimes \Sigma_n$ by using the ray structure (end structure) of the underlying set $X_n$. An eventual translation $g$ preserves each ray up to a finite set. Let $  R_i^*$ denote the set of all points of $R_i$ but finitely many. It is not difficult to see that if $\phi\in \Sym_n$ normalizes $\calH_n$ then
$$
(R_i^*)\phi= R_j^*\quad
$$
for $1 \leq i,j\leq n$. Thus $\phi$ defines an element $\sigma$ of $\Sigma_n$, and we see that $\phi \sigma^{-1} \in \calH_n$ since $(R_i^*)\phi\sigma^{-1}=(R_j^*)\sigma^{-1}=R_i^*$ for each $i$. Consequently, $N_{\Sym_n}(\calH_n)$ has the internal semidirect product of $\calH_n$ by $\Sigma_n$. Therefore we have:

\begin{Thm}[{\cite[Theorem~2.2]{BCMR}}]\label{auto}
For $n\ge2$, we have
$$
\aut(\calH_n)\cong\calH_n\rtimes \Sigma_n
$$
where $\Sigma_n$ is the symmetric group that permutes $n$ rays isometrically.
\end{Thm}

\section{The $R_\infty$ property for $\calH_n$, $n\ge2$}\label{sec:R number}

We consider the Houghton's groups $\calH_n$ with $n\ge2$.
Let $\phi$ be an automorphism on $\calH_n$. Remark that, when $n\ge3$,
$\phi$ induces an automorphism $\phi'$ on the commutator subgroup $\calH_n'=\FSym_n$
and an automorphism $\bar\phi$ on $\bbz^{n-1}$ so that
the following diagram is commutative:
$$
\CD
1@>>>\FSym_n@>{i}>>\calH_n@>{\pi}>>\bbz^{n-1}@>>>1\\
@.@VV{\phi'}V@VV{\phi}V@VV{\bar\phi}V\\
1@>>>\FSym_n@>{i}>>\calH_n@>{\pi}>>\bbz^{n-1}@>>>1
\endCD
$$
But when $n=2$, $\calH_2'=\FAlt_2$ and $\calH_2/\calH_2'=\bbz\oplus\bbz_2$.
Since $\FSym_2$ is a normal subgroup of $\calH_2$, we have the following commutative diagram
$$
\CD
@.@.1@.1\\
@.@.@AAA@AAA\\
@.@.\bbz@>=>>\bbz\\
@.@.@AAA@AAA\\
1@>>>\FAlt_2@>>>\calH_2@>>>\bbz\oplus\bbz_2@>>>1\\
@.@AA{=}A@AAA@AAA\\
1@>>>\FAlt_2@>>>\FSym_2@>>>\bbz_2@>>>1\\
@.@.@AAA@AAA\\
@.@.1@.1
\endCD
$$
Let $\phi \in\aut(\calH_2)$.
Then $\phi$ restricts to an element $\phi'$ of $\aut(\calH_2')=\aut(\FAlt_2)=\aut(\FSym_2)$,
and hence induces an automorphism $\bar\phi$ on $\bbz$ so that the following diagram is commutative
$$
\CD
1@>>>\FSym_2@>>>\calH_2@>>>\bbz@>>>1\\
@.@VV{\phi'}V@VV{\phi}V@VV{\bar\phi}V\\
1@>>>\FSym_2@>>>\calH_2@>>>\bbz@>>>1
\endCD
$$
These diagrams induce an exact sequence of Reidemeister sets
$$
\mathcal{R}[\phi']\buildrel{\hat{i}}\over\lra
\mathcal{R}[\phi]\buildrel{\hat{\pi}}\over\lra
\mathcal{R}[\bar{\phi}]\lra1.
$$
Because $\hat{\pi}$ is surjective, we have that if $R(\bar\phi)=\infty$, then $R(\phi)=\infty$.
Consequently, we have

\begin{Lemma}\label{lemma:infty}
Let $\phi$ be an automorphism on $\calH_n$, $(n\ge2)$.
If $R(\bar\phi)=\infty$, then $R(\phi)=\infty$.
\end{Lemma}

By Theorem~\ref{auto}, $\phi=\mu(\gamma\sigma)$ for some $\gamma\in\calH_n$ and $\sigma\in\Sigma_n$.
First, we will show that when $\phi=\mu(\sigma)$ for $\sigma \in \Sigma_n$ the Reidemeister number of $\phi$ is infinity.
When $\sigma=\id$, $\phi$ and hence $\bar\phi$ are identities. It is easy to see from definition that $R(\bar\phi)=R(\id)=\infty$, and so $R(\phi)=\infty$.

One useful observation in calculating $R(\mu(\sigma))$ is that a product
$\sigma=\sigma_1\sigma_2$ induces a bijection
\begin{equation}\label{eq:bijection}
R[\mu(\sigma_1)]\longleftrightarrow R[\mu(\sigma)],
\end{equation} which follows from
$$
b = h a h^{\sigma_1} \Leftrightarrow b \sigma_2 = h (a\sigma_2) h^{\sigma_1 \sigma_2}
$$ for all $a, b, h\in \calH_n$. Note that any product for $\sigma$ induces a bijection between the twist conjugacy classes of $\sigma$ and of the first term in the product. Recall that a cycle decomposition of a permutation $\sigma$ allows one to write $\sigma$ as a product of disjoint cycles. Since disjoint cycles commute there exists a bijection between $R[\mu(\sigma)]$ and $R[\mu(\sigma_1)]$ for any cycle $\sigma_1$ in a cycle decomposition of $\sigma$. The following observation plays a crucial role in the sequel.

\begin{Rmk}\label{rmk:bijection}
For a cycle $\sigma_1$ in a cycle decomposition of $\sigma\in \Sigma_n$, $R(\mu(\sigma_1))=\infty$ if and only if $R(\mu(\sigma))=\infty$.
\end{Rmk}


Recall that the \emph{cycle type} of a permutation $\tau\in \FSym_n$ encodes the data of how many cycles of each length are present in a cycle decomposition of $\tau$. Note that two permutations $\tau$ and $\tau'$ have the same cycle type if and only if they are conjugate in $\FSym_n$. In particular two cycles determine the same conjugacy class if and only if they have the same length. We extend this to establish a criterion for twisted conjugacy classes of cycles with respect to an automorphism $\phi= \mu(\sigma)$ when $\sigma \in \Sigma_n$ is a cycle.

\begin{Lemma}\label{lemma:sigma}
Suppose $\sigma\ne\id \in \Sigma_n$ is a cycle and $n\ge2$. A pair of cycles $\tau$ and $\tau'$ on the same ray determine the same twisted conjugacy class of $\phi=\mu(\sigma)$ if and only if they have the  equal length. In particular $R(\phi)= \infty$.
\end{Lemma}

\begin{proof}
Suppose that $\tau$ and $\tau'$ are cycles on the same ray of the equal length.
We first consider the case when $\sigma$ permutes rays
as an $\ell$-cycle $(1 \,2 \cdots \, \ell)$ for some $2\leq \ell\leq n$,
and $\tau$ and $\tau'$ are disjoint cycles on $R_1$.
Two cycles $\tau $ and $\tau'$ can be written as
$$
\tau=(p_1\cdots p_m) \text{ and }\tau'=(q_1 \cdots q_m)
$$
(by suppressing the ray notation) where $m\geq 2$.
We need to find an element $h\in\calH_n$ such that $\tau' = h\tau \mu(\sigma) (h)^{-1}$,
or equivalently
\begin{equation}\label{eq:class}
 h ^\sigma = \tau'^{-1} h \tau.
\end{equation} Let $h_1$ be the $2m$-cycle on $R_1$ given by
$$
h_1 = (p_1 q_1 \,p_2\, q_2 \cdots p_m\, q_m).
$$ It is direct to check that
\begin{equation}\label{eq:reduction}
\tau'^{-1} h_1 \tau = (q_m\cdots q_1)(p_1 q_1 \,p_2\, q_2 \cdots p_m\, q_m)(p_1 \cdots p_m)= h_1.
\end{equation}
Consider $h\in \calH_n'$ defined by
$$
h= h_1^{\sigma^{\ell-1}} \cdots h_1^\sigma h_1.
$$
Note that $h$ is a product of $\ell$ disjoint $2m$-cycles
each of which is an `isometric translation' of $h_1$ to the ray $R_{\ell}, \cdots, R_{2}, R_1$.
More precisely $(k+1,p) h_1^{\sigma^k} = (1,p)h_1 {\sigma^k}$ for all $(1,p) \in \supp(h_1)$
and $k=1, \cdots,\ell-1$. One crucial observation is that
\begin{equation*}
h^\sigma = h.
\end{equation*}
The above follows from that $\sigma$ is a $\ell$-cycle
and that components of $h$ have pairwise disjoint supports.
Moreover, $\tau'$ commutes with $h_1^{\sigma^{\ell-1}} \cdots h_1^\sigma$, so we have
\begin{align*}
h^\sigma &= h =
h_1^{\sigma^{\ell-1}} \cdots h_1^\sigma h_1=h_1^{\sigma^{\ell-1}} \cdots h_1^\sigma (\tau'^{-1} h_1 \tau )\\
&= \tau'^{-1}( h_1^{\sigma^{\ell-1}} \cdots h_1^\sigma h_1)\tau =\tau'^{-1}h \tau.
\end{align*}
Therefore $h$ satisfies the condition (\ref{eq:class}),
and hence $[\tau]=[\tau']$ in $R[\mu(\sigma)]$.

Applying appropriate conjugations one can extend the above observations to show that $[\tau]=[\tau']$ in $R[\mu(\sigma)]$ for any cycle $\sigma\in \Sigma_n$ and for any two disjoint cycles $\tau$ and $\tau'$ on the same ray with the equal length. Therefore, by the transitivity of the class, we can see that two cycles (not necessarily disjoint) on a ray belong to the same class for $\phi=\mu(\sigma)$ as long as they have the same length.
Indeed, if two $m$-cycles $\tau$ and $\tau'$ are not disjoint, one takes another $m$-cycle $\tau_0$ which is disjoint with $\tau $ and $\tau'$ to have $ [\tau]=[\tau_0]=[\tau']$. Thus we are done with one direction.

For the converse, suppose there exists $h\in \calH_n$ satisfying the condition (\ref{eq:class}) for a cycle $\sigma\in \Sigma_n$ even when cycles $\tau $ and $\tau'$ on the same ray have different lengths $m$ and $m'$ respectively. Assume $m' >m$. Let $\ell$ be the order of $\sigma$.
Applying the identity (\ref{eq:class}) $\ell$ times, we have
\begin{equation}\label{eq:iterated_conjugation}
h= h^{\sigma^\ell}= (\tau'^{-1})^{\sigma^{\ell-1}} \cdots ({\tau'^{-1}){^\sigma}} \tau'^{-1} h \tau {\tau^\sigma} \cdots \tau^{\sigma^{\ell-1}}.
\end{equation}
Let $c'=(\tau'^{-1})^{\sigma^{\ell-1}} \cdots (\tau'^{-1})^\sigma \tau'^{-1}$ and $c=\tau {\tau^\sigma} \cdots \tau^{\sigma^{\ell-1}} $ be the products of first and last $\ell$ terms on the RHS of (\ref{eq:iterated_conjugation}). Note that each component of $c'$ is an `isometric translation' of $\tau'^{-1}$ to different $\ell$ rays (and similarly for each component of $c$). To draw a contradiction, we use the fact that the size of $\supp(c')$ is strictly greater than that of $\supp(c)$. For details we need to examine how $h=c'hc$ acts on $\supp(c')$. Being a disjoint union, $\supp(c')=\bigcup_{0\leq k\leq \ell-1} (\supp(\tau'))\sigma^k$, $\supp(c')$ has size $\ell\times m'$, while $\supp(c)$ has size $\ell\times m$. For each $P \in \supp(c')$, we have
$$
(P)h=  (P)c' h c = (P')hc\;\; \text{  or  }\;\; (P)hc^{-1} = (P')h
$$
where $P'$ is a point in the same ray of $P$ but distinct from $P$. We claim that $(P)h$ belongs to $\supp(c)$. Otherwise $c^{-1}$ fixes $(P)h$, forcing $(P)h = (P')h$. Since $P\in \supp(c')$ was arbitrary, a bijection $h$ maps $\supp(c')$ to $\supp(c)$. We conclude that there does not exists $h\in \calH_n$ satisfying the condition (\ref{eq:class}) for cycles $\tau$ and $\tau'$ on the same ray with different lengths.
\end{proof}

\begin{Thm}\label{thm:conclusion_n}
The Houghton's groups $\calH_n$ have the $R_\infty$ property for all $n\ge2$.
\end{Thm}

\begin{proof}Theorem \ref{auto} says that an automorphism $\phi$ of $\calH_n$ is determined by $\phi=\mu(g\sigma)$ for some $g\in\calH_n$ and $\sigma \in \Sigma_n$. As we noted earlier, we may assume that $\sigma \neq 1$. Note that
$$
g\sigma=\sigma(\sigma^{-1}g\sigma)=\sigma g'
$$ with $g'\in\calH_n$. The product in RHS yields a bijection between $R[\mu(g\sigma)]$ and $R[\mu(\sigma)]$ as in (\ref{eq:bijection}). Consider a cycle $\sigma_1$ in a cycle decomposition of $\sigma$. Remark \ref{rmk:bijection} together with Lemma \ref{lemma:sigma} implies $R[\mu(\sigma)] = R[\mu(\sigma_1)]=\infty$. Therefore we have $R[\phi]=R[\mu(g\sigma)] = R[\mu(\sigma)]=\infty$ for all $\phi \in \aut(\calH_n)$ when $n\geq2$.
\end{proof}

We remark that Lemma~\ref{lemma:infty} can be used extensively to establish Theorem~\ref{thm:conclusion_n}. As observed in commuting diagrams above an automorphism $\phi=\mu(g\sigma)$ of $\calH_n$ induces an automorphism $\overline{\phi}$ on the abelianization $\bbz^{n-1}$, which is freely generated by $\pi(g_1),\cdots, \pi(g_{n-1})$. Since $\mu(g)$ fixes the generates $g_1, \cdots, g_{n-1}$, se wee that $\overline{\phi}= \mu(\sigma)$. The Reidemeister number of an automorphism on $\bbz^{n-1}$ ($n\geq 2$) is well understood. By  \cite[Theorem~6.11]{HLP}, $R(\overline{\phi})=\infty$ if and only if $\overline{\phi}$ has eigenvalue $1$. By using induction on $n$ one can show that $\overline{\phi}= \mu(\sigma)$ has eigenvalue $1$ unless $\sigma$ is an $n$-cycle on the rays $R_1, \cdots, R_n$. Now Lemma~\ref{lemma:sigma} implies that $R(\mu(\sigma))=\infty $ if $\sigma$ is a cycle, and so $R(\phi)=R(\overline{\phi})=\infty$.

\section{The group $\calH_1$ and its $R_\infty$ property}\label{sec:H_1}
In this section, we will study the $R_\infty$ property for the group $\calH_1$.
We remark that $\calH_1=\FSym_1$ is generated by the transpositions exchanging two consecutive points of $R_1$.
Let $\phi$ be an automorphism of $\calH_1$.
Since $\aut(\calH_1)=\aut(\FSym_1)\cong\Sym_1$, we have that $\phi=\mu(\gamma)$ for some $\gamma\in\Sym_1$.

\begin{Lemma}\label{lemma:image}
Let $\varphi:G\to G$ be an endomorphism. Then for any $g\in G$
we have $[g]=[\varphi(g)]$ in $R[\varphi]$.
\end{Lemma}

\begin{proof}
The lemma follows from
\begin{align*}
&\varphi(g)=(g^{-1})g\varphi(g^{-1})^{-1}.\qedhere
\end{align*}
\end{proof}

An infinite cycle $\gamma \in \Sym_1$ is given by a bijection $\gamma: \bbz \to R_1$. For convenience we use the $1$-to-$1$ correspondence to denote points of $\supp(\gamma)\subset R_1$ by integers, that is, each point of $\supp(\gamma)$ is denoted by its \emph{preimage}. With this notation, each infinite cycle can be realized as the translation on $\bbz$ by $+1$. Remark that if $h\in \FSym_1$ with $\supp(h) \subset \supp(\gamma)$ then the conjugation $\mu(\gamma)$ \emph{shifts} $\supp(h)$ to $\supp(h^\gamma)$ by $+1$;
\begin{equation}\label{eq:shift}
(k)h = k' \; \Leftrightarrow (k+1)h^\gamma = k'+1
\end{equation}
for all $k\in \supp(h)$. We say that an infinite cycle $\gamma$ \emph{conjugates} a permutation $\tau\in\FSym_1$ to $\tau'$ if $\tau'$ can be written as a conjugation of $\tau$ by a power of $\gamma$.

\begin{Lemma}\label{lemma:infinite cycle}
For an infinite cycle $\gamma\in \Sym_1$, two transpositions $\tau$ and $\tau'$ with $\supp(\tau) \subset \supp(\gamma)$ and $\supp(\tau')\subset \supp(\gamma)$ determine the same conjugacy class for $\phi=\mu(\gamma)$ if and only if $\gamma$ conjugates $\tau$ to $\tau'$. In particular $R(\mu(\gamma))= \infty$.
\end{Lemma}

\begin{proof}
Assume that $\tau'=\tau^{\gamma^m}$ or $\tau'=\phi^m(\tau)$, for some $m$.
By Lemma~\ref{lemma:image}, we have $[\tau]=[\phi(\tau)]=\cdots=[\phi^m(\tau)]=[\tau']$.

For the converse, suppose that there exists $h\in \FSym_1$ satisfying
\begin{equation}\label{eq:class_c}
h^\gamma =\tau'^{-1}h \tau=\tau'h\tau
\end{equation}
for two transpositions $\tau$ and $\tau'$ with the condition on their supports, one of which $\gamma$ does not conjugate to the other. By the shift (\ref{eq:shift}), they can be written as $\tau=(0\,\ell)$ and $\tau'=(m\,m\!+\!\ell')$ for some $m\ge0$ and $\ell\neq\ell'>0$.
By Lemma~\ref{lemma:image}, which implies $[(0\, \ell')]=[(m\,\, m\!+\!\ell')]$ for all $m\in \bbz$, we may further assume that $\tau'=(0\,\ell')$ and $\ell<\ell'$.

We first claim that $(-1)h=-1$. If $-1\in \supp(h)$, the identity (\ref{eq:class_c}) says
$$
(-1)h^\gamma = (-1)\tau' h\tau = (-1)h\tau \neq -1
$$since $\tau$ and $\tau'$ fix all negative integers. So $-1 \in \supp(h^\gamma)$. Now the shift
\begin{equation*}
k\in \supp(h) \;\Leftrightarrow\; k+1 \in \supp(h^\gamma)
\end{equation*}
implies $-2 \in \supp(h)$. Observe that the same argument establishes simultaneous induction on $k$ for
$$
-k \in \supp(h) \text{  and  } -k \in \supp(h^\gamma)
$$for all positive $k$ with the above base cases when $k=1$. This means that $\supp(h)$ must contain all negative integers. It contradicts that $h\in \FSym_1$. Therefore $h$ fixes $-1$, or equivalently $h^\gamma$ fixes $0$. One can also show $h$ fixes $\ell'+1$ by verifying
$$
\ell'+k \in \supp(h)\text{  and  }  \ell+k \in \supp(h^\gamma)
$$for all positive $k$ if we are given the base case $\ell'+1 \in \supp(h)$ (and $\ell'+1 \in \supp(h^\gamma)$, which follows immediately by (\ref{eq:class_c})). So we also have $(\ell'+1)h=(\ell'+1)$, and hence $(\ell'+1)h^\gamma=(\ell'+1)$ by (\ref{eq:class_c}).

From the fixed point $0=(0)h^\gamma$ we have
$$
(0)\tau' h \tau =0  \;\Leftrightarrow\; (\ell')h = \ell.
$$ The shift (\ref{eq:shift}) says $\ell'+1 \in \supp(h^\gamma)$. However this contradicts that $(\ell'+1)h^\gamma=(\ell'+1)$. Therefore $\tau' =(0\, \ell')$ does not belong to the class of $\tau=(0\, \ell)$ unless $\ell=\ell'$.
\end{proof}

\begin{Lemma}\label{lemma:finite_permutation}
Suppose two permutations $\tau,\tau' \in \FSym_1$ are disjoint with a permutation $\gamma\in \FSym_1$. Then $\tau$ and $\tau'$ belong to the same class in $R[\mu(\gamma)]$ if and only if they have the same cycle type. In particular $R(\mu(\gamma))= \infty$.
\end{Lemma}

\begin{proof}The statement follows from \emph{cycle type criterion} for usual conjugacy classes of the symmetric group on the fixed points of $\gamma\in \FSym_1$.
Any permutations on $R_1'=R_1\setminus\supp(\gamma)$ with finite supports are conjugate if and only if they have the same cycle type. For two permutations $\tau $ and $\tau'$ on $R_1'$ there exists a permutation $h \in \FSym_1$ on $R_1'$ such that
$$
\tau' = h \tau h^{-1}
$$if and only if $\tau$ and $\tau'$ have the same cycle type. Since $h^\gamma=h$ one can replace $h^{-1}$ by $(h^{-1})^\gamma$ in the identity to establish $\tau' = h \tau (h^{-1})^\gamma $.
\end{proof}

\begin{Thm}\label{thm:H1}
The group $\calH_1$ has the $R_\infty$ property.
\end{Thm}

\begin{proof}
Recall $\aut(\calH_1)=\aut(\FSym_1)\cong\Sym_1$. Each automorphism $\phi$ is given by $\phi = \mu(\gamma)$ for some $\gamma \in \Sym_1$. Consider the orbits of $\supp(\gamma)$ to form a partition of $\supp(\gamma)$. Observe that $\gamma$ restricts to a cycle on each orbit. Thus we see that a cycle decomposition of $\gamma$ is well defined and so $\gamma$ can be expressed as a product of commuting cycles.  If $\gamma$ has an infinite orbit then it contains an infinite cycle $\gamma_1$ so that $\gamma$ can be written as
\begin{equation}\label{eq:magic_product}
\gamma = \gamma_1 \gamma_2.
\end{equation} We have a bijection $R[\mu(\gamma_1)]\leftrightarrow R[\mu(\gamma)]$ from Remark \ref{rmk:bijection}. By Lemma \ref{lemma:infinite cycle}, we know that $R[\mu(\gamma_1)]=\infty$, and hence $R[\mu(\gamma)]=\infty$. If all orbits of $\gamma$ are finite then we can express $\gamma$ as a product (\ref{eq:magic_product}) with a finite cycle $\gamma_1$. From Lemma \ref{lemma:finite_permutation}, we see that $R[\mu(\gamma_1)]=\infty$, and so $R[\mu(\gamma)]=\infty$ due to the same bijection as above. We have proved that $R[\phi]=\infty$ for all automorphism of $\calH_1$.
\end{proof}

\bibliographystyle{siam}
\bibliography{R_inf_ref}

\end{document}